\documentclass{article}
\usepackage[letterpaper,margin=1in]{geometry}
\usepackage{times}

\usepackage[utf8]{inputenc}

\usepackage[
backend=biber,
style=alphabetic,
citestyle=alphabetic,
maxalphanames=6,
maxcitenames=99,
mincitenames=98,
maxbibnames=99,
giveninits=true,
]{biblatex}

\usepackage{amsmath}
\usepackage{amssymb}

\usepackage{amsthm}

\usepackage{lmodern}

\usepackage{bbm}
\usepackage{thm-restate}

\newtheorem*{theorem*}{Theorem}
\newtheorem{theorem}{Theorem}[section]
\newtheorem{lemma}[theorem]{Lemma}

\newtheorem{corollary}[theorem]{Corollary}

\newtheorem{observation}[theorem]{Observation}

\usepackage{url}
\usepackage[table]{xcolor}
\usepackage{array}

\usepackage{setspace}

\usepackage{wrapfig}

\makeatletter
\newcommand{\labeltarget}[1]{\Hy@raisedlink{\hypertarget{#1}{}}}
\makeatother

\usepackage{upgreek}

\usepackage{complexity}

\usepackage[inline]{enumitem}
\usepackage{moreenum}

\usepackage{mathtools}
\usepackage[super]{nth}

\usepackage{bm}

\usepackage{xspace}

\usepackage{ifthen}

\usepackage[immediate]{silence}
\usepackage{sectsty}
\DeactivateWarningFilters[tmp]
\allsectionsfont{\boldmath}

\setlist[enumerate]{nosep,topsep=0.1em}
\setlist[enumerate,1]{label=(\roman*), leftmargin=2.2em}
\setlist[itemize]{nosep,topsep=0.3em}

\usepackage[textsize=footnotesize, color=blue!30!white]{todonotes}

\usepackage{xfrac}

\usepackage{tikz}
\usetikzlibrary{calc}
\usetikzlibrary{shapes.geometric}
\usetikzlibrary{arrows}
\usetikzlibrary{decorations.pathreplacing, decorations.pathmorphing}

\usepackage{tcolorbox}
\tcbuselibrary{skins}

\usepackage{float}
\usepackage{graphicx}
\graphicspath{{graphics/}}
\makeatletter
\newcommand\appendtographicspath[1]{%
  \g@addto@macro\Ginput@path{#1}%
}
\makeatother

\usepackage[margin=10pt,font=small,labelfont=bf,skip=-5pt]{caption}
\usepackage{subcaption}

\usepackage[vlined,ruled,algo2e]{algorithm2e}
\SetAlgoSkip{bigskip}
\SetAlgoInsideSkip{smallskip}
\SetKwInput{KwInput}{Input}
\SetKwInput{KwOutput}{Output}

\usepackage{mdframed}

\usepackage{bigfoot}
\usepackage[pdfpagelabels,bookmarks=false]{hyperref}
\hypersetup{colorlinks, linkcolor=darkblue, citecolor=darkgreen, urlcolor=darkblue}

\usepackage[capitalize, nameinlink]{cleveref}

\definecolor{darkblue}{rgb}{0,0,0.38}
\definecolor{darkred}{rgb}{0.8,0,0}
\definecolor{darkgreen}{rgb}{0.1,0.35,0}

\DeclareMathOperator{\supp}{supp}

\DeclareMathOperator{\rank}{rank}

\renewcommand{\epsilon}{\varepsilon}
\renewcommand{\M}{\mathcal{M}}
\newcommand{\I}{\mathcal{I}}

\renewcommand{\E}{\mathbb{E}}
\renewcommand{\R}{\mathbb{R}}
\newcommand{\MOFS}{{\ensuremath{\mathrm{MOFS}}}\xspace}
\newcommand{\MSP}{{\ensuremath{\mathrm{MSP}}}\xspace}
\newcommand{\FMSP}{{\ensuremath{\mathrm{FMSP}}}\xspace}

\makeatletter
\let\@@pmod\pmod
\DeclareRobustCommand{\pmod}{\@ifstar\@pmods\@@pmod}
\def\@pmods#1{\mkern8mu({\operator@font mod}\mkern 6mu#1)}
\makeatother

\makeatletter
\let\@@mod\mod
\DeclareRobustCommand{\mod}{\@ifstar\@mods\@@mod}
\def\@mods#1{\mkern8mu{\operator@font mod}\mkern 6mu#1}
\makeatother

\makeatletter
\def\@fnsymbol#1{\ensuremath{\ifcase#1\or *\or 
\ddagger\or
    \mathsection\or \mathparagraph\or \|\or **\or \dagger\dagger
    \or \ddagger\ddagger \else\@ctrerr\fi}}
\makeatother

\title{Simple Random Order Contention Resolution for Graphic Matroids with Almost no Prior Information%
\thanks{This project received funding from Swiss National Science Foundation grant 200021\_184622 and the European Research Council (ERC) under the European Union's Horizon 2020 research and innovation programme (grant agreement No 817750).}
}

\author{
Richard Santiago\thanks{
Department of Mathematics, ETH Zurich, Zurich, Switzerland.
Email: \href{mailto:rtorres@ethz.ch}%
{rtorres@ethz.ch}.}
\and
Ivan Sergeev\thanks{
Department of Mathematics, ETH Zurich, Zurich, Switzerland.
Email: \href{mailto:isergeev@ethz.ch}%
{isergeev@ethz.ch}.}
\and
Rico Zenklusen\thanks{
Department of Mathematics, ETH Zurich, Zurich, Switzerland.
Email: \href{mailto:ricoz@ethz.ch}%
{ricoz@ethz.ch}.}
}

\date{}

\begin{document}

\maketitle

\begin{abstract}
Random order online contention resolution schemes (ROCRS) are structured online rounding algorithms with numerous applications and links to other well-known online selection problems, like the matroid secretary conjecture. We are interested in ROCRS subject to a matroid constraint, which is among the most studied constraint families. Previous ROCRS required to know upfront the full fractional point to be rounded as well as the matroid. It is unclear to what extent this is necessary. \citeauthor{fu2022oblivious} (SOSA 2022) shed some light on this question by proving that no strong (constant-selectable) online or even offline contention resolution scheme exists if the fractional point is unknown, not even for graphic matroids.

In contrast, we show, in a setting with slightly more knowledge and where the fractional point reveals one by one, that there is hope to obtain strong ROCRS by providing a simple constant-selectable ROCRS for graphic matroids that only requires to know the size of the ground set in advance. Moreover, our procedure holds in the more general adversarial order with a sample setting, where, after sampling a random constant fraction of the elements, all remaining (non-sampled) elements may come in adversarial order.
\end{abstract}

\begin{tikzpicture}[overlay, remember picture, shift = {(current page.south east)}]
\begin{scope}[shift={(-1.1,1.5)}]
\def\hd{2.5}
\node at (-2.15*\hd,0) {\includegraphics[height=0.7cm]{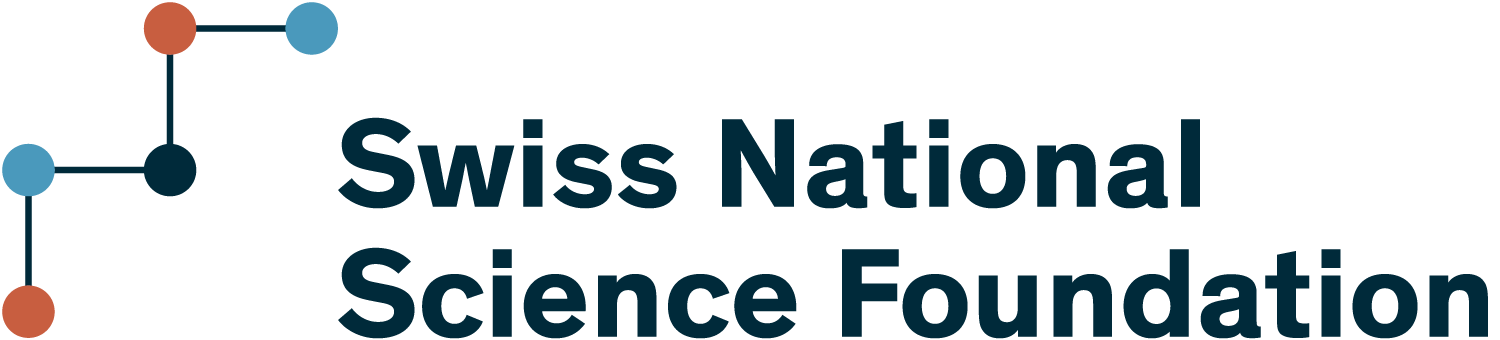}};
\node at (-\hd,0) {\includegraphics[height=1.0cm]{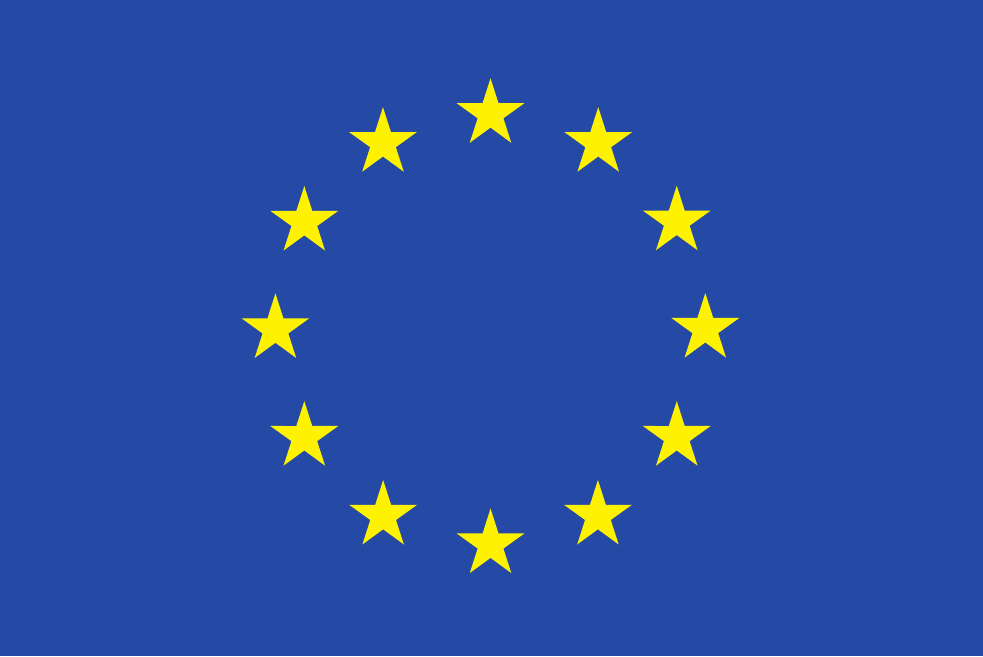}};
\node at (-0.2*\hd,0) {\includegraphics[height=1.2cm]{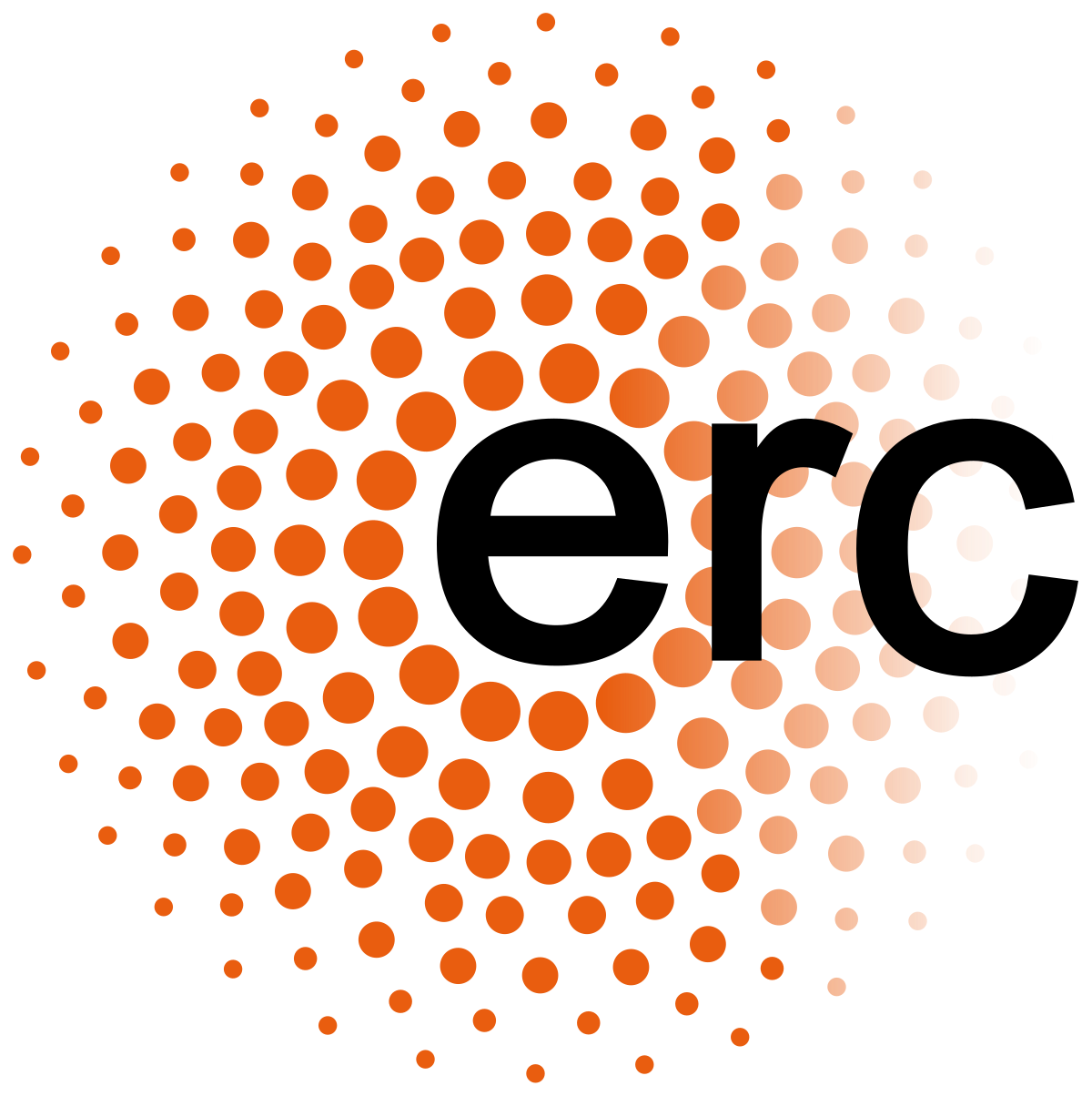}};
\end{scope}
\end{tikzpicture}

\section{Introduction}

Contention resolution schemes (CRS) are a class of (typically randomized) rounding algorithms to transform a fractional point from a polyhedral relaxation into an integral one while preserving feasibility. They were introduced in \cite{chekuri2014submodular} motivated by applications to submodular maximization problems, and have proven to be a versatile tool. Contention resolution schemes were subsequently extended to the online setting~\cite{feldman2016online}, where elements arrive one by one in an online fashion. (See also~\cite{adamczyk2018random}.) They have found applications in areas such as stochastic probing, posted pricing mechanisms, and prophet inequalities. (For further information on these areas, we refer the interested reader to the following recent works and references therein: \cite{adamczyk2016, gupta2013, gupta2016, fu2021} for stochastic probing, \cite{chawla2010, yan2011, kleinberg2012} for posted pricing mechanisms, and \cite{ezra2020online, kleinberg2012, kleinberg2019} for prophet inequalities.)

Formally, we are given a downwards-closed family $\I \subseteq 2^N$, a polyhedral relaxation $P_{\I}$ thereof,%
\footnote{We call a polyhedron $P_{\I} \subseteq [0,1]^N$ a \emph{polyhedral relaxation} of $\I$ if $P_{\I}\cap \{0,1\}^N = \{\mathbf{1}_S \colon S \in \I\}$, where $\mathbf{1}_S\in \{0,1\}^N$ is the characteristic vector of $S$. Hence, the $0/1$-points in $P_{\I}$ correspond to $\I$.}
and a vector $x \in P_{\I}$. Elements of $N$ arrive one by one. Each element $e\in N$ is either active, with probability $x_e$, or inactive, with probability $1-x_e$, independently of the other elements. Whether an element is active is revealed upon its arrival. Moreover, whenever an arriving element $e\in N$ is active, the algorithm must decide irrevocably whether to pick it or not. At all times, the set of currently selected elements $T$ must satisfy $T \in \I$. The goal is to select each element $e$ with probability at least $c \cdot x_e$ for $c\in [0,1]$ as large as possible. The scheme is then called \emph{$c$-selectable}.%
\footnote{We note that the term \emph{$c$-balanced} has also been used in the literature instead of \emph{$c$-selectable}.}
Different models have been studied depending on the arrival order of the elements. Random order contention resolution schemes (ROCRS) \cite{adamczyk2018random, lee2018optimal} assume that elements are revealed in uniformly random order, while online contention resolution schemes (OCRS) \cite{alaei2014bayesian, feldman2016online, lee2018optimal} assume the order is fixed by an adversary, typically upfront.%
\footnote{Different types of adversaries have been considered in the literature, such as offline, online, and almighty adversaries.}

While offline and online CRS have been studied under several types of constraints, in this work we restrict our attention to matroid constraints, which is among the most studied constraint classes in this context.%
\footnote{A matroid $\mathcal{M}$ is a pair $\mathcal{M}=(E,\I)$ where $E$ is a finite set and $\I \subseteq 2^E$ is a non-empty family satisfying: 1) if $A \subseteq B$ and $B \in \I$ then $A \in \I$, and 2) if $A,B \in \I$ and $|B|>|A|$ then $\exists e \in B \setminus A$ such that $A \cup \{e\} \in \I$.}
For matroids it is known that contention resolution schemes with constant-factor guarantees (i.e., $c=\Omega(1)$) exist. They have been developed for the offline setting \cite{chekuri2014submodular}, the random order setting \cite{adamczyk2018random, lee2018optimal}, and the online (adversarial) setting \cite{feldman2016online, lee2018optimal}.

The above models, however, usually assume prior knowledge of both the matroid $\M$ and the point $x$ to be rounded, where $x$ is required to be in the matroid polytope $P_{\M}$.%
\footnote{The matroid polytope $P_{\M}\subseteq [0,1]^E$ of matroid $\M=(E,\I)$ is the convex hull of all characteristic vectors of independent sets.}
This prior knowledge is exploited heavily when building and analyzing the respective schemes. A key question is how much prior knowledge is required to obtain strong guarantees (i.e., constant selectability). This question is motivated by related problems in the area where existing procedures seem to need too much information to have a chance to be useful for progress on further known questions. (In \cref{sec:connectionMSP} we further expand on this using a connection to the matroid secretary problem.) In terms of prior knowledge, the following two questions capture central aspects:
\begin{itemize}
    \item How much information about the point $x$ to be rounded is known in advance?
    \item How much of the matroid $\mathcal{M}$ is known in advance?
\end{itemize}
In terms of knowledge of $x$, the arguably strongest notion of (distribution) obliviousness is that the CRS does not depend on $x$ whatsoever (and hence, also does not learn $x$ over time). Such a CRS is said to be \emph{oblivious}. Hence, whenever an element arrives, the algorithm only learns whether that element is active. In \cite{chekuri2014submodular}, an oblivious $\Omega(1)$-selectable CRS for laminar matroids was presented. This raised the question to what extent such schemes may exist more generally. Very recently, \cite{fu2022oblivious} showed that, unfortunately, oblivious $\Omega(1)$-selectable CRS (and hence selectable OCRS) do not exist for all matroids; more precisely neither graphic nor transversal matroids admit them. In addition, they provide an oblivious $\frac{1}{e}$-selectable OCRS for uniform rank-$1$ matroids.

Our main goal is to show that, despite this negative result, there is hope to obtain strong ROCRS without having to know the full point $x$ upfront. We say that a ROCRS is \emph{distribution unknown} if the point $x \in P_{\M}$ is not known upfront but rather revealed as elements appear. That is, when an element $e$ is revealed, we learn whether it is active together with its marginal value $x_e$. Note that this is less restrictive than the above-mentioned oblivious setting, where no information on $x$ is gained or used throughout the algorithm. We find the assumption of having access to $x_e$ natural as it resembles the input one gets in an online rounding procedure. In addition, our interest for the problems above partially stems from connections to the matroid secretary problem (and its fractional variant), where having access to the $x_e$ values plays an important role. We discuss this in more detail in \cref{sec:connectionMSP}.

Moreover, regarding prior knowledge about the matroid, we call a ROCRS \emph{matroid unknown} if it only knows the cardinality of the matroid upfront, and the matroid is revealed as elements appear.%
\footnote{Knowing the cardinality upfront is a very common assumption and necessary to have any hopes for constant-selectable algorithms to exists. As we see later, this is also implied by a stronger impossibility result of \cite{fu2022oblivious}.}
Finally, we call a ROCRS \emph{distribution and matroid unknown} if it is both distribution and matroid unknown.

\subsection{Connections to the matroid secretary problem}\label{sec:connectionMSP}

An instance of the matroid secretary problem \MSP \cite{babaioff2018matroid} consists of a matroid $\M=(E,\I)$ and weights $w\colon E \to \R_{\geq 0}$. The elements of the matroid are revealed online in uniformly random order. Once an element $e\in E$ arrives, we see its weight $w_e$ and must decide irrevocably whether to pick it or not. The algorithm must always output an independent set $S \in \I$, and the goal is to maximize (in expectation) its total weight $w(S)\coloneqq \sum_{e\in S} w_e$. It is a major open question whether MSP admits a constant-factor competitive algorithm, and the \emph{matroid secretary conjecture} claims this to be the case.

Interestingly, the matroid secretary conjecture remains open even for sparse weight functions, where the elements $\supp(w)\coloneqq \{e\in E \colon w(e) > 0\}$ can be partitioned into constantly many independent sets, or even just two independent sets. A very related way to phrase the same problem is to assume that we are given a matroid whose ground set is the union of $k$ independent sets, however the algorithm does not know the matroid upfront. We denote by $\MSP(k)$ this restricted sparse version of MSP. Clearly, if the matroid was known upfront (and $k$ is constant), it would be trivial to obtain a $k$-competitive algorithm for $\MSP(k)$ by first partitioning the ground set into $k$ independent sets and selecting uniformly at random all elements of one of these sets, independently of the revealed weights. We highlight that not knowing the matroid upfront is a natural assumption, as one can obfuscate an instance by adding $0$-weight elements to the matroid.%
\footnote{Note that this is not a formal proof of equivalence between the MSP conjecture with known or unknown matroid.
However, especially in general matroids, which is what the MSP conjecture is about, there are many ways to obfuscate an instance by adding $0$-weight elements, which, even if the matroid is known upfront, makes it much harder to exploit this.}
This is also reflected by the currently best $\Omega(\log \log (\rank))$-competitive algorithms for MSP for general matroids~\cite{lachish_2014_competitive,feldman_2018_simple}, which do not need to know the matroid upfront. Interestingly, despite the existence of numerous constant-competitive procedures for MSP for specific classes of matroids~\cite{%
babaioff2018matroid,%
korula_2009_algorithms,%
im_2011_secretary,%
soto_2013_matroid,%
jaillet_2013_advances,%
ma_2013_simulatedSTACS,%
dimitrov_2012_competitive,%
kesselheim_2013_optimal,%
dinitz_2013_matroid,%
},
even for graphic matroids (and other non-trivial matroid classes), no constant-competitive procedure is known for $\MSP(2)$ when only the cardinality of the matroid is known in advance and one has access to an independence oracle for elements revealed so far. The reason is that MSP procedures for specific matroid classes typically assume either full knowledge of the matroid upfront or that revealing elements also reveal additional information about the matroid, like the part of an explicit representation of the matroid corresponding to the revealed element.

Note that a matroid unknown constant-selectable ROCRS solves $\MSP(k)$ for constant $k$ because we can run the ROCRS with the uniform point $x=(\frac{1}{k},\frac{1}{k},\ldots,\frac{1}{k})$. As the ground set of the matroid can be partitioned into $k$ independent sets, the point $x$ lies in the matroid polytope. The ROCRS being constant-selectable implies that each element will be picked with probability $\sfrac{c}{k}$ for some constant $c \in [0,1]$, thus selecting each element with constant probability. In other words, a matroid unknown constant-selectable ROCRS can solve the following \emph{matroid online fair selection} problem (\MOFS), which implies a constant-selectable algorithm for $\MSP(k)$ (for constant $k$) and, so we think, is interesting in its own. Assume we are given a matroid $M$ whose ground set is the union of $k$ independent sets. The matroid is unknown to the algorithm upfront and reveals one element at a time. Whenever an element reveals, the algorithm has to decide irrevocably whether to accept it. The task is to select each element with probability $\Omega(\sfrac{1}{k})$.

We summarize the connections among the problems introduced so far in \cref{fig:problems}.

\begin{figure}[ht]
    \centering
    \begin{tikzpicture}[xscale=2]
        \node (dmu) at (0, 0) {Distribution and matroid unknown ROCRS};
        \node (du) at (-1.5, -1) {Distribution unknown ROCRS};
        \node (mu) at (1.5, -1) {Matroid unknown ROCRS};
        \node (std) at (0, -2) {(Standard) ROCRS};
        \node (mop) at (2, -2) {MOFS};
        \node (msp) at (2, -3) {Sparse MSP};
        \node (spacing) at (2, -3.5) {};

        \draw[->] ([xshift=-5pt]dmu.south) -- (du.north);
        \draw[->] ([xshift=5pt]dmu.south) -- (mu.north);
        \draw[->] (du.south) -- ([xshift=-5pt]std.north);
        \draw[->] ([xshift=-5pt]mu.south) -- ([xshift=5pt]std.north);
        \draw[->] ([xshift=5pt]mu.south) -- (mop.north);
        \draw[->] (mop.south) -- (msp.north);
    \end{tikzpicture}
    \caption{Relations between mentioned problems. An arrow from setting $\mathrm{A}$ to setting $\mathrm{B}$ indicates that a constant-selectable/constant-competitive algorithm for $\mathrm{A}$ implies one for $\mathrm{B}$.}\label{fig:problems}
\end{figure}
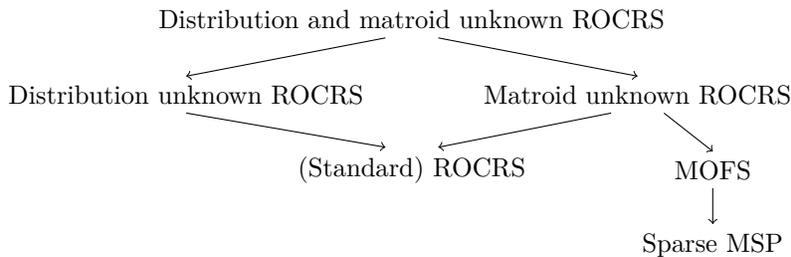

Another related problem is the \emph{fractional matroid secretary problem} (\FMSP).%
\footnote{While this problem is known in the community, we are not aware of any previous work on it.}
An instance of the latter is almost identical to one of \MSP, the only difference being that one can pick elements \emph{fractionally}, in the sense that the output is a vector $y \in P_{\M}$. The goal is still to maximize the expected weight of the output, in this case given by $w^\top y$. Fractional \MSP is clearly no harder than classical \MSP. However, even for fractional \MSP, it is unknown whether a constant-competitive algorithm exists. 

One arguably natural approach to solve the \MSP conjecture is to first solve the (potentially easier) fractional version, and then look for online rounding procedures that transform an \FMSP solution into an (integral) \MSP one, while losing only a constant factor of the objective value. One may wonder whether a distribution and matroid unknown $\Omega(1)$-selectable ROCRS immediately implies the existence of such a rounding procedure. Unfortunately the answer is no. The reason is that the output of an \FMSP algorithm may highly depend on the arrival order of the input (i.e., two different arrival orders may lead to outputs $y_1,y_2 \in P_{\M}$ where $y_1$ is not a permutation of $y_2$), while a ROCRS assumes the hidden distribution $x$ is fixed beforehand, and then revealed to the algorithm in random order. However, any such rounding procedure that allows for online rounding an \FMSP solution into an \MSP one is a ROCRS, and, in particular, also provides a solution to $\MOFS$ (and thus to sparse $\MSP$ as well).

\subsection{Our results}\label{sec:results}

Our main result is a simple $\Omega(1)$-selectable distribution and matroid unknown ROCRS for graphic matroids, where we assume that, whenever an edge reveals, it reveals its endpoints. We allow graphs to have parallel edges, and we assume without loss of generality that they do not contain loops. The procedure only needs to know the cardinality of the ground set upfront, i.e., the number of edges of the graph, so that it can sample a constant fraction of the elements. We emphasize that we predominantly focus on providing a simple algorithm and analysis, and make no attempt at optimizing the selectability constant.

\begin{theorem}\label{thm:main}
    Let $G=(V,E)$ be a graph, $\M=(E,\I)$ its graphic matroid, and $P_{\M}$ the corresponding matroid polytope. Then there exists a distribution and matroid unknown $\frac{1}{96}$-selectable random order contention resolution scheme for $P_{\M}$.
\end{theorem}

Like many other procedures in this context, ours has first an observation phase, which samples a set $S \subseteq E$ containing a constant fraction of the elements without picking any. We use this to learn some structure of the underlying instance, and then run a second phase where elements are selected, which is the core of our procedure. Our analysis of the second phase does not use the values of the marginals $x_e$ from $e \in E\setminus S$, nor the assumption that elements from $E\setminus S$ arrive in uniformly random order. Thus, our procedure and analysis still hold in the more general \emph{adversarial order with a sample} setting, where we are allowed to sample a random constant fraction of the elements and the remaining (non-sampled) elements arrive in adversarial order. Moreover, for non-sampled elements, we only need to know whether they are active or not without needing to learn their $x$-values.

As discussed above, \cref{thm:main} immediately implies the following.
\begin{corollary}\label{cor:MOPresult}
    There is a $\frac{1}{96}$-selectable algorithm for \MOFS on graphic matroids. 
\end{corollary}

\section{Algorithm}\label{sec:algorithm}

Let us first present a canonical way to obtain a constant-selectable procedure if we know the graph $G$ and the vector $x\in P_{\mathcal{M}}$ upfront. This is a nice warm-up before we present our ROCRS and, moreover, it allows us to highlight important hurdles that appear in the matroid and distribution unknown setting when this prior information is not available.

\subsection{Warm up: a simple constant-selectable procedure with prior knowledge}\label{sec:priorKnowledge}

We define a total order on the vertices iteratively, by first identifying the smallest vertex in the ordering, then the next smallest one and so on. The ordering is obtained by successively picking a vertex $v$ for which $x(\delta(v))$, i.e., the sum of the $x$-values on edges incident with $v$, is smallest, and then deleting it and iterating. More formally, let $Q\subseteq V$ be the vertices ordered so far. (Hence, we have $Q=\emptyset$ at the beginning.) As the next vertex we pick one minimizing the $x$-load of edges to vertices in $V\setminus Q$, i.e., we pick a minimizer of
\begin{equation*}
    \min_{v\in V\setminus Q} x(E(v,V\setminus Q)),
\end{equation*}
where $E(v,V\setminus Q)$ are all edges with one endpoint being $v$ and the other one belonging to $V\setminus Q$. We call such an ordering an $x$-topological ordering and write $v \prec u$ for $v, u \in V$ if $v$ is ordered before $u$.

Note that the sets
\begin{equation*}
    E_{v}^{\prec} \coloneqq \{e=\{v,u\}\in \delta(v) \colon v\prec u\},
\end{equation*}
partition all edges $E$. A constant-selectable random order contention resolution scheme is now obtained by selecting an appearing edge $e = \{v,u\} \in E$ (say $e\in E_{v}^{\prec}$) according to the following rule. If $e$ is active and no other edge of $E_{v}^{\prec}$ has been selected so far, then select $e$ with probability $\sfrac{1}{8}$; otherwise, do not select $e$. One can think of this as flipping upfront a biased coin for each edge and declaring it \emph{excluded} with probability $\sfrac{7}{8}$. We then simply select greedily the first active and non-excluded edge for each $E_{v}^{\prec}$.

First note that this will return a forest due to the following well-known and simple observation.
\begin{observation}\label{obs:returnedSetIsForest}\crefformat{observation}{#2Observation~#1#3}
    Let $G = (V, E)$ be a graph with a total vertex order '$\prec$', and let $T\subseteq E$ with $|T\cap E_{v}^{\prec}|\leq 1$ for $v\in V$. Then $T$ is a forest.
\end{observation}
Indeed, the observation holds because any cycle $C$ in $G$ must contain at least $2$ edges from $E_{v}^{\prec}$ where $v$ is the first vertex of $C$ with respect to the vertex order '$\prec$'.

The crucial property we need of an $x$-topological ordering to get constant selectability is
\begin{equation}\label{eq:evSmall}
    x(E_{v}^{\prec}) \leq 2 \qquad \forall v\in V.
\end{equation}
To see that this property holds, let $v\in V$, and we denote by $Q\coloneqq E_v^{\prec}$ the vertices considered before $v$ in the construction of the $x$-topological ordering '$\prec$'. Observe first that $x(E[V\setminus Q])\leq |V\setminus Q|-1$, where $E[V\setminus Q]$ are all edges with both endpoints in $V\setminus Q$, because $x$ is in the forest polytope. Indeed, no forest can contain more than $|V\setminus Q|-1$ edges with both endpoints in $V\setminus Q$, which immediately bounds the $x$-value of any convex combination of forests. Because $2 x(E[W]) = \sum_{u\in W} x(E(u, W))$ for any $W \subseteq V$, which can be interpreted as a fractional version of the handshaking lemma, we have
\begin{equation*}
    \sum_{u\in V\setminus Q} x(E(u, V\setminus Q)) = x(E[V\setminus Q]) \leq 2|V\setminus Q| - 2,
\end{equation*}
which implies~\eqref{eq:evSmall} by an averaging argument, because $v$ minimizes $x(E[u,V\setminus Q])$ among all vertices $u\in V\setminus Q$.

Now consider the probability of selecting an edge $e=\{v,u\}\in E$. Edge $e$ will certainly be selected if the following two conditions hold simultaneously:
\begin{enumerate}
    \item $e$ is active and not excluded.
    \item\label{item:xKnownEdgesExcludedOrNonActive} Each edge in $(E_{v}^{\prec} \cup E_{u}^{\prec}) \setminus \{e\}$ is either excluded or not active.
\end{enumerate}
Observe that these two events are independent, the first one happens with probability $\sfrac{x_e}{8}$, and the second one with constant probability. To see the latter, note that the expected number of active edges in $E_{v}^{\prec} \cup E_{u}^{\prec}$ is $x(E_{v}^{\prec} \cup E_{u}^{\prec}) \leq x(E_{v}^{\prec}) + x(E_{u}^{\prec}) \leq 4$ due to~\eqref{eq:evSmall} and hence
\begin{equation}\label{eq:xKnownExpBadEdges}
    \E[|\{f\in (E_{v}^{\prec} \cup E_{u}^{\prec})\setminus \{e\} \colon f \text{ is active and not excluded}\}|] = \frac{1}{8} \E[|\{f\in (E_{v}^{\prec} \cup E_{u}^{\prec}) \setminus \{e\} \colon f \text{ is active} \}|] \leq \frac{1}{2}.
\end{equation}
Then applying Markov's inequality to~\eqref{eq:xKnownExpBadEdges} immediately implies that event~\ref{item:xKnownEdgesExcludedOrNonActive} happens with probability at least $\frac{1}{2}$. Thus the contention resolution scheme has constant selectability. (More precisely, the above reasonings imply a selectability of $\sfrac{1}{16}$.)

\subsection{Our procedure}

We show that a simple and canonical extension of the prior knowledge algorithm discussed in~\cref{sec:priorKnowledge} to our setting without this prior knowledge leads to a constant-selectable ROCRS. Although the extended algorithm may seem straightforward at first, the key challenge lies in its analysis. More precisely, we first observe a sample $S \subseteq E$ of edges and define an $x$-topological ordering of the subgraph over only the edges of $S$ and only using the $x$-values of those edges (which are known after sampling $S$). We denote the obtained ordering '$\prec_S$', where the subscript highlights that the $x$-topological ordering now depends on the sample $S$. Then we simply run the same algorithm as in \cref{sec:priorKnowledge} (albeit with a different probability to exclude edges) on the remaining edges $E \setminus S$ using the ordering '$\prec_S$'. The resulting procedure is presented in \cref{alg:vrp}.
\begin{algorithm2e}[ht]
    \caption{Distribution and matroid unknown ROCRS for graphic matroids.}\label{alg:vrp}

    \DontPrintSemicolon
    \SetKwFunction{FnVRP}{VRP}

    \SetKwInput{KwInput}{Information known upfront}
    \SetKwInput{KwOnlineInfo}{Information revealed online}
    \KwInput{$m = |E|$}
    \KwOnlineInfo{For each arriving $e \in E$: its endpoints, $x_e$, and whether $e$ is active}
    \KwOutput{Forest $F \subseteq E$}

    Sample $s \sim \mathrm{Binom}(m, \frac{1}{2})$, observe the set $S \subseteq E$ of the first $s$ arriving edges\;

    Let $\prec_{S}$ be the $x$-topological ordering of $V$ with respect to the sample $S$\;

    \For{every $e = \{v, u\} \in E \setminus S$ arriving online}{
        Without loss of generality assume that $v \prec_{S} u$ (otherwise swap $v$ and $u$)\;
        \If{$e$ is active and no edges $\{v, t\}$ with $v \prec_{S} t$ were previously picked}{
            Pick $e$ with probability $\frac{1}{24}$
        }
    }
\end{algorithm2e}
Although this algorithm naturally extends the offline procedure, several important questions arise regarding both its implementation and, most importantly, its analysis. Indeed, even though we can mimic the definition of $E^\prec_v$ from \cref{sec:priorKnowledge} and introduce a partitioning of $E$ induced by the order '$\prec_{S}$' via
\[
    E_{v}^{\prec_S} = \{e = \{v, u\} \in \delta(v) \colon v \prec_{S} u\},
\]
it is no longer true that the partitioning defined this way satisfies property~\eqref{eq:evSmall}. (In fact, $x(E_{v}^{\prec_S})$ can even be super-constant depending on the sampled set $S$.) Nevertheless, note that \cref{alg:vrp} clearly returns a forest due to \cref{obs:returnedSetIsForest}. We discuss the implementation details first and continue with the analysis of the algorithm later in~\cref{sec:analysis}. For ease of notation, we use the shorthand $E_v^S$ for $E_v^{\prec_S}$.

Let us outline the issues that need to be addressed to complete the description of \cref{alg:vrp}. First, recall that here, in contrast to the prior knowledge setting, we cannot access the whole vertex set upfront. Thus, our algorithm can only define the $x$-topological order for the vertices that have been seen in the sampling phase (as endpoints of edges in $S$) and has to maintain the ordering as more vertices are revealed over time. In addition, while the procedure from \cref{sec:priorKnowledge} could break ties between minimizers of $x(\delta(v))$ arbitrarily, in our setting, we handle tiebreaking more carefully, as it may have a non-trivial impact on the resulting partition and hence the selectability of our algorithm, especially due to the randomness of the sample set $S$.

First let us address the questions related to the $x$-topological ordering by modifying the naive procedure. To this end, assume for now that a tiebreaking rule is fixed; we will discuss how to properly implement tiebreaking later. Now, let $T$ be the set of all endpoints of the edges in the sample set $S$ and note that we have full access to $T$ once $S$ is observed. With that, we can initialize $G'$ to be the edge-induced subgraph $G[S] = (T, S)$ of $G$ and apply the usual iterative process to $G'$ in order to define the $x$-topological ordering on $T$. The remaining vertices (i.e., those not in $T$) can be added to the $x$-topological ordering once we observe them for the first time (as an endpoint of an arriving edge from $E \setminus S$) and, in accordance with the offline method of defining the ordering, such vertices come before $T$ in the ordering and are ordered amongst themselves as per the tiebreaking rule. More formally, the first time our algorithm receives an edge incident to a vertex $v \in V \setminus T$, we set $v \prec_{S} t$ for all $t \in T$ and either $v \prec_{S} u$ or $u \prec_{S} v$ according to the tiebreaking rule for all other previously observed $u \in V \setminus T$. Thus, given a fixed tiebreaking rule, the $x$-topological ordering is uniquely defined in the distribution and matroid unknown setting for all vertices.

Now it only remains to address how to break ties whenever they occur. The property we seek form a tiebreaking rule is that it is ``consistent'', in the sense that it does not depend on the arrival order of the edges. If the vertex set $V$ was revealed upfront, a canonical way of consistent tiebreaking would be to arbitrarily label the vertices $v_{1}, v_{2}, \dots, v_{n}$ at the start; then during the procedure we could break ties by choosing the vertex $v_{j}$ with the lowest index $j$ among all the candidates. However, this method needs to know $V$ in advance. We overcome this issue by generating a uniformly random labeling on the fly as follows. Throughout the procedure, maintain a labeling of the vertices seen so far; whenever a new vertex is revealed, insert it into the labeling at a uniformly random position (and shift the labels of the consequent vertices). It is not hard to see that this method is equivalent to using a uniformly random vertex labeling that is fixed upfront, but has the added benefit of not requiring any prior knowledge of the vertex set. In particular, this implementation of tiebreaking can be used in the distribution and matroid unknown setting. Moreover, it is clearly a consistent tiebreaking rule. Going forward, we assume that a consistent tiebreaking rule (i.e., the corresponding vertex labeling) is fixed and used when defining the order '$\prec_S$' in \cref{alg:vrp}.

\section{Analysis}\label{sec:analysis}

We recall that the reason why the analysis of the algorithm in \cref{sec:priorKnowledge} does not work anymore is because property \eqref{eq:evSmall} fails for the sets
\begin{equation*}
    E^S_v\coloneqq \{e\in \{v,u\}\in \delta(v)\colon v\prec_S u\},
\end{equation*}
for the ordering '$\prec_S$' constructed in \cref{alg:vrp}. More precisely, in general, $x(E^S_v)$ cannot be bounded by a constant. The bound from \cref{sec:priorKnowledge} now only applies to the sampled edges (and holds for any realization of $S$), i.e.,
\begin{equation}\label{eq:degBoundSOnly}
    x \Big( \big\{ e = \{v, u\} \in \delta(v) \cap S \colon v \prec_{S} u \big\} \Big) \leq 2\enspace.
\end{equation}
Our main technical contribution is to show that the \emph{expected} $x$-values of non-sampled edges in $E^S_v$ can be bounded.
\begin{theorem}\label{thm:boundedExpXVal}
    The $x$-topological ordering of $V$ with respect to the sample $S$ constructed in \cref{alg:vrp} satisfies
    \begin{equation*}
        \E_S [x(E^S_v \setminus S)] \le 3 \qquad \forall v\in V.
    \end{equation*}
\end{theorem}

We first observe that this is indeed enough to obtain our main theorem.
To this end we follow the same proof strategy as in \cref{sec:priorKnowledge}, where we had prior knowledge of $x$.
The main difference is, because $x(E_v^S\setminus S)$ is small in expectation, we use a Markov bound to show that with constant likelihood both endpoints of a fixed edge $e$ do not contain any other incident edges (except for $e$) that are both active and non-excluded.
\begin{proof}[Proof of \cref{thm:main}]
    We already discussed that \cref{alg:vrp} returns a forest. Hence, it remains to show that each edge $e$ is selected with probability at least $\sfrac{x_{e}}{96}$. To this end, let $F$ denote the edges selected by the ROCRS and let $e = \{v, u\}$ be an arbitrary edge. Note that a sufficient condition for $e$ to be selected by \cref{alg:vrp} is that the following three events $A$, $B$, and $C$ hold simultaneously:
    \begin{itemize}
        \item[$A$:] $e$ is active and non-excluded.
        \item[$B$:] Each edge in $(E_v^S\cup E_u^S)\setminus (S\cup \{e\})$ is either excluded or inactive.
        \item[$C$:] $e\not\in S$.
    \end{itemize}
    Note that event $A$ is independent of $B$ and $C$ as it only depends on two coin flips for edge $e$, one deciding whether it is active and one deciding whether it is excluded. These coin flips are independent of the set $S$ and coin flips of other elements. Thus we have
    \begin{equation}\label{eq:breakingDownABC}
        \Pr[A \wedge B \wedge C] = \Pr[A] \cdot \Pr[B\wedge C] \geq \Pr[A] \cdot (1 - \Pr[\neg B] - \Pr[\neg C]),
    \end{equation}
    where the second inequality follows from a union bound. Let $D_e^S \coloneqq (E_v^S \cup E_u^S)\setminus (S\cup \{e\})$. Due to \cref{thm:boundedExpXVal}, the expected number of active edges within $D_e^S$ is upper bounded by $6$. Because each element is excluded with probability $\sfrac{23}{24}$, the expected number of active and non-excluded elements of $D_e^S$ is bounded by $\sfrac{1}{4}$. Hence, by Markov's inequality, we have $\Pr[\neg B] \leq \frac{1}{4}$. Putting this into~\eqref{eq:breakingDownABC} together with $\Pr[A]=\sfrac{x_e}{24}$ and $\Pr[C]=\sfrac{1}{2}$, we get $\Pr[A\wedge B \wedge C] \geq \sfrac{x_e}{96}$ as desired.
\end{proof}

It remains to show the main technical result, \cref{thm:boundedExpXVal}. We begin by providing the intuition behind our proof approach. Ideally, we would like to link the $x$-load of vertices in $E_v^S\setminus S$ to the one of vertices in $E_v^S\cap S$. Indeed, because we constructed the ordering $\prec_S$ with respect to the edges in $E_v^S\cap S$, we have $x(E_v^S\cap S)\leq 2$, by the same reasoning as in \cref{sec:priorKnowledge}. A nice and arguably natural property one could wish to have to link these quantities is that for every node $v \in V$ and every incident edge $e = \{v, u\} \in \delta(v)$ the following inequality holds:
\[
    \Pr_{S}[(e \notin S) \wedge (v \prec_{S} u)] \le \Pr_{S}[(e \in S) \wedge (v \prec_{S} u)].
\]
With the above \emph{coupling}-type property, one would be done immediately, because
\begin{align*}
    \E_S [x(E^S_v \setminus S)]
     = \smashoperator[r]{\sum_{e = \{v, u\} \in \delta(v)}} x_{e} \Pr[(e \notin S) \wedge (v \prec_{S} u)]
    \le \smashoperator[r]{\sum_{e = \{v, u\} \in \delta(v)}} x_{e} \Pr[(e \in S) \wedge (v \prec_{S} u)]
     = \E_S [x(E^S_v \cap S)] \le 2.
\end{align*}
Unfortunately, the above property is too strong and does not hold in general, even for small simple instances. We provide such an example in~\cref{sec:appendix}.

We show that a similar, though weaker, coupling property holds, which is nevertheless sufficient to obtain the desired bound on the load of the sets $E_v^S \setminus S$. The key ingredient is to replace the set of events $\{S\subseteq E: (e \in S) \wedge (v \prec_{S} u) \}$ considered above by a larger family of events given by $ \{ S \subseteq E: (e \in S) \wedge (w_S(v) \preceq_{S} u) \}$, where $w_S(v)$ denotes a special node satisfying $w_S(v) \preceq_S v$, to be formally defined later. The next result highlights the two crucial properties we want $w_S(v)$ to have.
\begin{theorem}\label{thm:coupling}
    Let $v \in V$ be an arbitrary node, let $S \subseteq E$ be an arbitrary sample set, and let $\prec_S$ denote the $x$-topological ordering of $V$ with respect to the sample $S$ constructed in \cref{alg:vrp}. Then there is a vertex $w_S(v) \in V$ with the following properties:
    \begin{itemize}\setlength\itemsep{3pt}
        \item $x \big(\{e = \{v, u\} \in \delta(v) \cap S \colon w_{S}(v) \preceq_{S} u\} \big) \le 3$, and
        \item $\Pr_S \left[(e \notin S) \wedge (v \prec_{S} u)\right] \le \Pr_S \left[(e \in S) \wedge (w_{S}(v) \preceq_{S} u)\right]$ for any edge $e=\{v,u\} \in \delta(v)$.
    \end{itemize}
\end{theorem}

\noindent Assuming the above result, the desired bound on the load of the partitions is then immediate.

\begin{proof}[Proof of \cref{thm:boundedExpXVal}]
    We have
    \begin{align*}
        \E_S [x(E^S_v \setminus S)]
        & = \sum_{e = \{v, u\} \in \delta(v)} x_{e} \Pr[(e \notin S) \wedge (v \prec_{S} u)]
        \le \sum_{e = \{v, u\} \in \delta(v)} x_{e} \Pr[(e \in S) \wedge (w_{S}(v) \preceq_{S} u)] \\
        & = \E_S [x(\{e = \{v, u\} \in \delta(v) \cap S \colon w_{S}(v) \preceq_{S} u\})] \le 3. \qedhere
    \end{align*}
\end{proof}

Thus, all that remains is to prove~\cref{thm:coupling}. To do so, we need the following additional result, which describes how a given topological ordering $\prec_{S}$ can change when an edge $e \in E \setminus S$ is added to the sample $S$.

\begin{lemma}\label{lem:low-rk}
    Given a sample $S \subsetneq E$ and an edge $e = \{v, u\} \notin S$, let $\tilde{S} = S \cup \{e\}$ and consider the $x$-topological orderings $\prec_{S}$ and $\prec_{\tilde{S}}$. Then for all $t \in V$ such that $t \prec_{S} v$ and $t \prec_{S} u$ we have $t \prec_{\tilde{S}} v$ and $t \prec_{\tilde{S}} u$.
\end{lemma}
\begin{proof}
    Note that adding $e$ to the subgraph $G[S]$ induced by $S$ can only increase $x(\delta_{G[S]}(v))$ and $x(\delta_{G[S]}(u))$. Since additionally the tiebreaking rule is fixed, the procedures for determining the $x$-topological ordering starting with $G[S]$ and $G[\tilde{S}]$ will proceed exactly the same as long as there remains a vertex $t$ such that $t \prec_{S} v$ and $t \prec_{S} u$.
\end{proof}

We are now ready to prove \cref{thm:coupling}, which completes the proof of our main result, \cref{thm:main}.

\begin{proof}[Proof of~\cref{thm:coupling}.]
    Given a sample set $S \subseteq E$ and a vertex $v \in V$, we define $w_S(v) \in V$ as follows:
    \begin{itemize}
        \item if $x(\delta(v) \cap S) \le 2$, let $w_{S}(v)$ be the vertex such that $w_{S}(v) \preceq_{S} u$ for all $u \in V$;
        \item if $x(\delta(v) \cap S) > 2$, let $w_{S}(v)$ be the vertex such that after the procedure for determining the $x$-topological ordering deletes it, the remaining total $x$-value of the edges incident to $v$ becomes at most $2$.
    \end{itemize}
    Note that if $x(\delta(v) \cap S) \le 2$, then the first statement of the theorem holds by the definition of $w_{S}(v)$. On the other hand, if $x(\delta(v) \cap S) > 2$, then the definition of $w_{S}(v)$ implies that until the procedure for determining the $x$-topological ordering deletes it from the graph, the vertex $v$ cannot be removed. In particular, $w_{S}(v) \prec_{S} v$. Moreover, by the definition of $w_{S}(v)$ we have
    \begin{align}
        x(\{e = \{v, u\} \in \delta(v) \cap S \colon w_{S}(v) \prec_{S} u\}) & \le 2 \nonumber \\
        2 < x(\{e = \{v, u\} \in \delta(v) \cap S \colon w_{S}(v) \preceq_{S} u\}) & \le 3, \label{eq:wproperty}
    \end{align}
    where the last inequality follows from the fact that the total $x$-load of all edges with endpoints $v$ and $w_S(v)$ is at most one, because $x$ is in the forest polytope. (Note the difference in the two expressions displayed above in terms of '$\prec_S$' vs.~'$\preceq_S$'.) This concludes the proof of the first statement of the theorem.

    Next we prove the second statement. Observe that since every realization of $S$ is equiprobable and $S$ contains each edge with probability $\sfrac{1}{2}$ independently, to prove the desired inequality it suffices to show that there exists an injective map from the set of realizations of $S$ satisfying $e \notin S$ and $v \prec_{S} u$ to the set of realizations satisfying $e \in S$ and $w(v) \preceq_{S} u$. To this end, consider any realization of $S$ such that $e \notin S$ and $v \prec_{S} u$ and the injective map $S \mapsto \tilde{S} \coloneqq S \cup \{e\}$. We show that $w_{\tilde{S}}(v) \preceq_{\tilde{S}} u$ by considering the following three cases.
    \begin{itemize}
        \item Suppose $x(\delta(v) \cap \tilde{S}) \le 2$. Then $w_{\tilde{S}}(v) \preceq_{\tilde{S}} u$ holds by the definition of $w_{\tilde{S}}(v)$.

        \item Suppose $x(\delta(v) \cap \tilde{S}) > 2$ and $x(\delta(v) \cap S) \le 2$. Additionally, for the sake of deriving a contradiction suppose that $u \prec_{\tilde{S}} w_{\tilde{S}}(v)$. Then
        \begin{align*}
            x \Big( \big\{ e' = \{v, t\} \in \delta(v) \cap \tilde{S} \colon w_{\tilde{S}}(v) \preceq_{\tilde{S}} t \big\} \Big)
                \le x \big(\delta(v) \cap (\tilde{S} \setminus \{e\}) \big)
                = x \big(\delta(v) \cap S \big) \le 2,
        \end{align*}
        which contradicts the definition of $w_{\tilde{S}}(v)$, namely the strict inequality in \eqref{eq:wproperty}. Thus $w_{\tilde{S}}(v) \preceq_{\tilde{S}} u$.

        \item Suppose $x(\delta(v) \cap S) > 2$. Additionally, for the sake of deriving a contradiction suppose that $u \prec_{\tilde{S}} w_{\tilde{S}}(v)$. Our goal is to show that then $x(\{e'=\{v,t\}\in \delta(v)\cap \tilde{S} \colon w_{\tilde{S}}(v) \preceq_{\tilde{S}} t\}) \leq 2$, which violates the strict inequality in~\eqref{eq:wproperty}.
        First, note that, since $u \prec_{\tilde{S}} w_{\tilde{S}}(v)$ by our assumption, we have
        \begin{equation}\label{eq:xLoadAfterWST}
            x \Big( \big\{ e' = \{v, t\} \in \delta(v) \cap \tilde{S} \colon w_{\tilde{S}}(v) \preceq_{\tilde{S}} t \big\} \Big)
              =  x \Big( \big\{ e' = \{v, t\} \in \delta(v) \cap S \colon w_{\tilde{S}}(v) \preceq_{\tilde{S}} t \big\} \Big).
        \end{equation}
        Second, observe that for any $t\in V$ we have
        \begin{equation}\label{eq:afterWSTIsAfterV}
            w_{\tilde{S}}(v) \preceq_{\tilde{S}} t \implies v \preceq_{S} t.
        \end{equation}
        Indeed, the contrapositive holds because if $t \prec_{S} v$, then, by \cref{lem:low-rk}, we have $t \prec_{\tilde{S}} u$; combining this with our assumption $u \prec_{\tilde{S}} w_{\tilde{S}}(v)$ implies \cref{eq:afterWSTIsAfterV}. Finally, we obtain the desired contradiction by combining the above observations:
        \begin{align*}
            x \Big( \big\{ e' = \{v, t\} \in \delta(v) \cap \tilde{S} \colon w_{\tilde{S}}(v) \preceq_{\tilde{S}} t \big\} \Big)
                & \leq x \Big( \big\{ e' = \{v, t\} \in \delta(v) \cap S \colon v \preceq_{S} t \big\} \Big)\\
                &   =  x \Big( \big\{ e' = \{v, t\} \in \delta(v) \cap S \colon v \prec_{S} t \big\} \Big)\\
                & \leq 2.
        \end{align*}
        Here the first inequality holds by~\eqref{eq:xLoadAfterWST} and~\eqref{eq:afterWSTIsAfterV}, the equality holds as $v \neq t$, and the last inequality holds by~\eqref{eq:degBoundSOnly}. \qedhere
    \end{itemize}
\end{proof}

Finally, we note that the argument used in our proof of \cref{thm:main} is robust with respect to the arrival order of the edges in $E\setminus S$. Indeed, if the events $A$, $B$, and $C$ hold, then $e$ is selected in any arrival order of the elements in $E\setminus S$. Hence, \cref{thm:main} holds even in the setting where the algorithm first observes a sample $S \subseteq E$ containing each edge independently with probability $\sfrac{1}{2}$ and the remaining edges arrive in adversarial order. Hence, our algorithm also works in the sample-based OCRS setting with at least the same selectability.

\section{Conclusion}

In this paper we presented a constant-selectable random order contention resolution scheme for graphic matroids, in the setting where the matroid and the distribution are not given upfront, but rather revealed as elements arrive. In fact, our algorithm works in the more restrictive setting where a constant fraction of the elements of the matroid is given upfront as a random sample and the remaining elements arrive in adversarial order. In contrast to previous work, we are able to considerably reduce the amount of prior information required by the algorithm while still achieving constant selectability.

We believe that the distribution and/or matroid unknown ROCRS setting introduced in this paper opens many avenues for future work. For instance, an immediate natural question is whether one can also obtain distribution and matroid unknown ROCRS for other classes of matroids admitting an explicit representation via (hyper) graphs, such as transversal matroids, hypergraphic matroids, and gammoids.  On the more general side, it remains an interesting question whether general matroids admit distribution and/or matroid unknown constant selectable ROCRS. Ultimately, developing online contention resolution schemes with fewer and fewer assumptions about the input might eventually lead to an algorithmic reduction from fractional MSP to MSP.

\printbibliography

\appendix
\section{Appendix}\label{sec:appendix}

In this section we discuss why the inequality $\Pr_{S}[(e \notin S) \wedge (v \prec_{S} u)] \le \Pr_{S}[(e \in S) \wedge (v \prec_{S} u)]$ where $e = \{v, u\} \in E$, which was discussed at the beginning of~\cref{sec:analysis}, does not hold in general.

We begin by considering the basic instance shown in~\cref{fig:counterexample-basic}, where the $x$-value associated to each edge is indicated next to it on the figure. Observe that if we compute the $x$-topological ordering for sample sets $S_{1} = E \setminus \{e\}$ and $S_{2} = E$, we get $v \prec_{S_{1}} u$ and $u \prec_{S_{2}} v$, respectively, regardless of which tiebreaking rule is used. This observation shows that even if a sample set $S \subseteq E$ satisfies $e \notin S$ and $v \prec_{S} u$, then $v \prec_{S \cup \{e\}} u$ does not necessarily hold.

\begin{figure}[ht]
    \centering
    \begin{tikzpicture}[xscale=2]
        \begin{scope}[every node/.style={circle,draw,inner sep=3pt}]
            \node (a) at (0, 0) {$a$};
            \node (b) at (1, 0) {$b$};
            \node (v) at (2, 0) {$v$};
            \node (u) at (3, 0) {$u$};
            \node (c) at (4, 0) {$c$};
        \end{scope}
        \node (space) at (0, -1) { };

        \draw[-] (a) -- node[above] {$1.0$} (b);
        \draw[-] (b) -- node[above] {$0.1$} (v);
        \draw[-] (v) -- node[above] {$0.5$} (u);
        \draw[-] (u) -- node[above] {$0.4$} (c);
    \end{tikzpicture}
    \caption{Basic instance for constructing examples where $\Pr_{S}[(e \notin S) \wedge (v \prec_{S} u)] > \Pr_{S}[(e \in S) \wedge (v \prec_{S} u)]$.}
    \label{fig:counterexample-basic}
\end{figure}
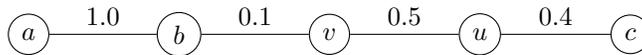

However, we are not quite done yet: one can manually verify that the basic instance shown in~\cref{fig:counterexample-basic} yields $\Pr_{S}[(\{v, u\} \notin S) \wedge (v \prec_{S} u)] > \Pr_{S}[(\{v, u\} \in S) \wedge (v \prec_{S} u)]$ only for one of two ways to break ties between $v$ and $u$, while the other tiebreaking rule instead results in an equality. In contrast, we are interested in counterexamples that do not rely on a particular tiebreaking rule to work. Nevertheless, such counterexamples can be obtained via a modification of the basic instance as we will now demonstrate. Intuitively, if we modify the basic instance so that with high probability the sampled set looks ``similar" to either $S_{1}$ or $S_{2}$ described above, then, roughly speaking, we will almost always end up in the case where $v \prec_{S \setminus \{e\}} u$ and $u \prec_{S \cup \{e\}} v$ and as a result the original inequality will be broken for any fixed tiebreaking rule. To this end, we introduce the following modification: replace each edge $e' \in E \setminus \{e\}$ with $k \in \mathbb{Z}_{\geq 1}$ copies of it with associated $x$-values equal to $\frac{x(e')}{k}$. We denote the resulting instance $G_{k}$ and use $E_{k}(a, b)$, $E_{k}(b, v)$, and $E_{k}(u, c)$ to denote the sets of edges in $G_{k}$ between $a$ and $b$, $b$ and $v$, and $u$ and $c$, respectively. It is not difficult to see that for sufficiently large $k$ the Chernoff concentration bound and the union bound ensure that with high probability we simultaneously have $x(E_{k}(a, b) \cap S) \in [0.41, 1]$, $x(E_{k}(b, v) \cap S) \in [0.01, 0.1]$, and $x(E_{k}(u, c) \cap S) \in [0.11, 0.4]$. Conditioned on this event occurring, we get $v \prec_{S \setminus \{e\}} u$ and $u \prec_{S \cup \{e\}} v$ as desired.
Consequently, for sufficiently large $k$, the multigraph $G_{k}$ is indeed a counterexample to the original inequality irrespective of the chosen tiebreaking rule.

While the above argument already disproves the original inequality, we additionally provide a smaller counterexample instance in~\cref{fig:counterexample-specific}, which avoids the step where we replace edges by parallel copies and use concentration bounds.
This example can be verified manually by computing the $x$-topological ordering for all $32$ possible realizations of $S$. 
More specifically, for realizations of $S$ with $\{v, u\} \notin S$, it turns out that $v \prec_{S} u$ in $13$ cases, while in $2$ further cases the ordering depends on the tiebreaking rule. As for realizations of $S$ with $\{v, u\} \in S$, we get $u \prec_{S} v$ in $5$ cases and in the remaining $11$ cases the ordering depends on the tiebreaking rule. Therefore, regardless of the chosen tiebreaking rule, we get $\Pr_{S}[(\{v, u\} \notin S) \wedge (v \prec_{S} u)] > \Pr_{S}[(\{v, u\} \in S) \wedge (v \prec_{S} u)]$ for this instance.

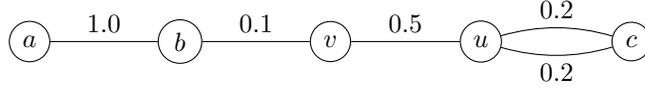
\begin{figure}[ht]
    \centering
    \begin{tikzpicture}[xscale=2]
        \begin{scope}[every node/.style={circle,draw,inner sep=3pt}]
            \node (a) at (0, 0) {$a$};
            \node (b) at (1, 0) {$b$};
            \node (v) at (2, 0) {$v$};
            \node (u) at (3, 0) {$u$};
            \node (c) at (4, 0) {$c$};
        \end{scope}
        \node (space) at (0, -1) { };

        \draw[-] (a) -- node[above] {$1.0$} (b);
        \draw[-] (b) -- node[above] {$0.1$} (v);
        \draw[-] (v) -- node[above] {$0.5$} (u);
        \draw[-] (u) edge[bend left] node[above] {$0.2$} (c);
        \draw[-] (u) edge[bend right] node[below] {$0.2$} (c);
    \end{tikzpicture}
    \caption{Example instance for which $\Pr_{S}[(\{v, u\} \notin S) \wedge (v \prec_{S} u)] > \Pr_{S}[(\{v, u\} \in S) \wedge (v \prec_{S} u)]$ holds.}
    \label{fig:counterexample-specific}
\end{figure}

\end{document}